\documentclass{article}
\usepackage{tikz, pgfplots, wrapfig}
\usepackage{caption}
\usepackage{subcaption}
\usepackage{amssymb}
\usepackage{amsfonts}
\usepackage{amsmath}
\usepackage{fullpage}
\usepackage{graphicx}
\usepackage[utf8]{inputenc}
\usepackage{amsthm}
\usepackage{textcomp}

\usepackage{xfrac}

\newtheorem{theorem}{Theorem}

\newtheorem{definition}[theorem]{Definition}

\newtheorem{lemma}[theorem]{Lemma}

\newtheorem{proposition}[theorem]{Proposition}

\newcommand{\rP}{\mathrm{P}} 
\newcommand{\rE}{\mathrm{E}} 

\begin{document}

\title{
Rank-dependent Galton-Watson processes and their pathwise duals}
\author{Serik Sagitov\thanks{%
Chalmers University and University of Gothenburg, 412 96 Gothenburg, Sweden. Email address:
serik@chalmers.se. 
} \ and Jonas Jagers
}
\maketitle

\begin{abstract}
We introduce a modified Galton-Watson process using the framework of an infinite system of particles labeled by $(x,t)$, where $x$ is the rank of the particle born at time $t$. The key assumption concerning the offspring numbers of different particles is that they are independent, but their distributions may depend on the particle label $(x,t)$. For the associated system of coupled monotone Markov chains, we address the issue of pathwise duality elucidated by a remarkable graphical representation, with the trajectories of the primary Markov chains and their duals coalescing together to form forest graphs on a two-dimensional  grid. 
\end{abstract}

\section{Introduction}

The Galton-Watson (GW) process is a basic stochastic model for the generation size for a population of reproducing particles, see  \cite{AN}. Slightly modifying the framework of  \cite{Be}, we define a  GW process in terms of an infinite system of particles uniquely labeled by pairs $(x,t)\in\mathbb N\times\mathbb Z$, where $t$ refers to the generation number and $x$ is the {\it rank} of the particle within this generation. Given a set of independent and identically distributed random variables
\begin{equation}\label{ut} 
\big\{u_{t}(x)\big\}_{(x,t)\in\mathbb N\times\mathbb Z}
\end{equation}  
taking values in  $\mathbb N_0=\{0\}\cup\mathbb N$, a GW process stemming from $Z_a$ particles at time $a\in \mathbb Z$, is the Markov chain $\{Z_t\}_{t\ge a}$ characterized by the branching property
\begin{equation}\label{Yt} 
Z_{t+1}=\sum_{x=1}^{Z_t} u_t(x),
\end{equation}  
with $u_t(x)$ representing the offspring number of the particle $(t,x)$. 
Relation \eqref{Yt} induces the following rank-inheritance rules: 

(i) each particle $(x,t+1)$ has a unique parent $(x',t)$, 

(ii) if $x<y$, then $x'<y'$, where $(x',t)$ and  $(y',t)$ are the parents  of  $(x,t+1)$ and  $(y,t+1)$. 

\noindent For example, if $u_t(1)=k$ is positive, then $k$ children of the rank 1 particle get the ranks $1,\ldots,k$ among the particles born at time $t+1$. The ranks of particles play no role in the  standard GW  setting, however, they were used  in  \cite{Be} studying the GW processes with neutral mutations.

This paper introduces  a new modification of the GW  model by allowing the rank of a particle to determine its reproduction law.  
In a general rank-dependent GW setting, the independent offspring numbers $u_{t}(x)$ have  distributions
that vary over the birth times $t$ and particle ranks $x$. 
To illustrate, consider a linear-fractional reproduction law
\begin{align*}
\rE s^{u_t(x)}&=1-q_t(x)+q_t(x){p_ts\over 1-(1-p_t)s},  \quad p_t\in(0,1],  \quad q_t(x)=\left\{
\begin{array}{cl}
  1,&\text{ if }x=1, 3,\ldots,   \\
  0, &  \text{ if }x=2,4,\ldots,
\end{array}
\right.
\end{align*}
where the dependence on the particle rank takes effect via $q_t(x)$, the probability  of having non-zero offspring. Here, the particles with odd ranks always produce  $k\ge1$ offspring with probability $(1-p_t)^{k-1}p_t$, while the particles of even ranks have no offspring. 
Notice that the corresponding rank-dependent GW process can not be treated as a two-type GW process, since the number of even-ranked children for the rank 3 particle depends on the number children of the rank 1 particle.

 The standard GW process has many extensions, usually motivated by biological applications, see \cite{HJV, KA}. Some of these extensions can be viewed as examples of rank-dependent GW processes, see  Section \ref{BD}, where  the scope of the rank-dependent GW setting is highlighted by referring to bounded GW processes,  GW processes with immigration and emigration, duals to birth-death GW processes in varying environment, as well as GW processes embedded in continuous time linear birth-death processes  in varying environment.
 In particular, if the reproduction law  $\rE s^{u_t(x)}=f_{t}$ is not influenced by the particle rank, then the rank-dependent GW process is a  GW process in varying environment satisfying
\begin{equation*}
\rE (s^{Z_t}|Z_a=z)= (f_{a}\circ \ldots\circ   f_{t-1}(s))^z,
\end{equation*}
where $f\circ g(s)$ stands for  $f(g(s))$, see   \cite{Ja}, as well as  \cite{BS,BH,Ker} for recent treatments involving this model. In the rank-dependent GW setting, the last relation does hold in general,  making  analysis more complicated. 

The main results of the paper are collected in Section \ref{Sdu}. Our Theorem \ref{proM} considers the rank-dependent GW processes along with their pathwise dual processes,  whose definition in Section \ref{se} is based on Siegmund's duality, see \cite{JK,Si,StS}. It shows in particular, that the dual to the dual of a rank-dependent GW process is a shifted copy of the original rank-dependent GW process. In the literature on dual processes, the common setting involves time-homogeneous Markov processes. A notable exception is \cite{AS} treating a class of stationary processes. Our approach handles time-inhomogeneous Markov chains, and can even be adapted to the non-Markov setting, when for example, the offspring number $u_t(x)$  depends on the  offspring number $u_{t-1}(x')$ of the parent.

The infinite particle system framework allows for an illuminating graphical representation of a system of coupled rank-dependent GW processes and their pathwise duals visualizing their trajectories as forest graphs. 
 A process dual to an asexual reproduction model, like GW process or Wright-Fisher model, is usually interpreted in the terms of a coalescent model \cite{GH, Mo}. 
Somewhat counter-intuitively, our  graphical representation says that the dual to a branching process is again a form of the branching process with dependencies, see Figure \ref{Fig1}. 
The graphical representation works also for the primary reproduction models with fixed population size, like the Wright-Fisher model.

One of the examples  in Section \ref{BD} shows that, even with a standard GW process, the dual Markov chain is not necessarily a rank-dependent GW process, because the dual offspring numbers become dependent on each other. 
An interesting open problem is to characterize the class of rank-dependent GW processes, whose dual Markov chain is itself a  rank-dependent GW process. 
A simpler problem is  to characterize the class of GW processes, whose dual Markov chain is itself a  rank-dependent GW process. We obtain two results addressing the latter question. Consider the dual of the GW reproduction law.
Proposition \ref{JJ} says that the marginal dual offspring distribution is always linear-fractional.
Theorem \ref{th} demonstrates that the dual process is GW with an eternal particle if and only if that the primary reproduction law is itself linear-fractional.
Yet another example  in Section \ref{BD}  demonstrates that the dual to a GW process might be  a rank-dependent GW process which is not a GW with an eternal particle.

Section \ref{Spro} contains the proofs of the results stated in Section \ref{Sdu}.

\begin{figure*}[t]
    \centering
    \begin{subfigure}[t]{0.32\textwidth}
         \begin{tikzpicture}[scale=0.6]

\draw[step=1cm,dotted, thin] (-0,0) grid (7,7);

\node[black] () at (-0.5,0){$-3$};
\node[black] () at (-0.5,1){$-2$};
\node[black] () at (-0.5,2){$-1$};
\node[black] () at (-0.3,3){$0$};
\node[black] () at (-0.3,4){$1$};
\node[black] () at (-0.3,5){$2$};
\node[black] () at (-0.3,6){$3$};
\node[black] () at (-0.3,7){$4$};

\node[black] () at (0,-0.4){$0$};
\node[black] () at (1,-0.4){$1$};
\node[black] () at (2,-0.4){$2$};
\node[black] () at (3,-0.4){$3$};
\node[black] () at (4,-0.4){$4$};
\node[black] () at (5,-0.4){$5$};
\node[black] () at (6,-0.4){$6$};
\node[black] () at (7,-0.4){$7$};

\draw [line width=0.2mm,red] (0,0) -- (0,7);

\draw [line width=0.2mm,red] (1,7.03) -- (1,6.97);
\draw [line width=0.2mm,red] (2,7.03) -- (2,6.97);
\draw [line width=0.2mm,red] (3,7.03) -- (3,6.97);
\draw [line width=0.2mm,red] (6,7.03) -- (6,6.97);
\draw [line width=0.2mm,red] (7,7.03) -- (7,6.97);

\draw [line width=0.2mm,red]  (5,7)-- (6,6)-- (6,5);
\draw [line width=0.2mm,red]  (4,7)-- (5,6);
\draw [line width=0.2mm,red] (2,7) --(4,6)--(4,5);
\draw [line width=0.2mm,red]  (2,7)-- (3,6)-- (3,5);
\draw [line width=0.2mm,red]  (2,7)-- (2,6);
\draw [line width=0.2mm,red]  (0,7)-- (1,6);

\draw [line width=0.2mm,red]  (7,6)-- (7,5);
\draw [line width=0.2mm,red]  (6,6)-- (5,5);
\draw [line width=0.2mm,red]  (3,6)-- (2,5);
\draw [line width=0.2mm,red]  (3,6)-- (1,5)-- (3,4);

\draw [line width=0.2mm,red]  (6,5)-- (7,4)-- (7,3)-- (7,2);
\draw [line width=0.2mm,red]  (4,5)-- (6,4)-- (5,3);
\draw [line width=0.2mm,red]  (4,5)-- (5,4);
\draw [line width=0.2mm,red] (3,5)--(4,4)--(4,3);
\draw [line width=0.2mm,red]  (1,5)-- (2,4)-- (2,3);
\draw [line width=0.2mm,red]  (0,5)-- (1,4);

\draw [line width=0.2mm,red]  (7,4)-- (6,3)-- (6,2);
\draw [line width=0.2mm,red]  (4,4)-- (3,3)-- (3,2);
\draw [line width=0.2mm,red]  (0,4)--(1,3);

\draw [line width=0.2mm,red]  (5,3)-- (5,2);
\draw [line width=0.2mm,red]  (5,3)-- (4,2);
\draw [line width=0.2mm,red]  (3,3)-- (2,2)-- (3,1);
\draw [line width=0.2mm,red]  (2,3)-- (1,2);

\draw [line width=0.2mm,red]  (2,2)-- (4,1);
\draw [line width=0.2mm,red]  (4,2)-- (5,1)-- (4,0);
\draw [line width=0.2mm,red]  (4,2)-- (6,1);
\draw [line width=0.2mm,red]  (2,2)--(2,1);
\draw [line width=0.2mm,red]  (0,2)-- (1,1)--(1,0);

\draw [line width=0.2mm,red] (7,0.5) -- (6,0);
\draw [line width=0.2mm,red]  (7,1)-- (5,0);
\draw [line width=0.2mm,red]  (3,1)-- (3,0);
\draw [line width=0.2mm,red]  (3,1)-- (2,0);

\end{tikzpicture}
   
        \caption{}
    \end{subfigure}%
    ~ 
    \begin{subfigure}[t]{0.32\textwidth}
         \begin{tikzpicture}[scale=0.6]

\draw[step=1cm,dotted, thin] (0,0) grid (7,7);

\draw [line width=0.2mm,black] (3,-0.03) -- (3,0.03);
\draw [line width=0.2mm,black] (7,-0.03) -- (7,0.03);

\draw [line width=0.2mm,black] (0,0) -- (0,7);
\draw [line width=0.2mm,black] (1,0) -- (1,1);
\draw [line width=0.2mm,black] (2,0) -- (2,1) -- (1,2);
\draw [line width=0.2mm,black] (2,0) -- (3,1);
\draw [line width=0.2mm,black] (4,0) -- (4,1);
\draw [line width=0.2mm,black] (4,0) -- (5,1) -- (3,2);
\draw [line width=0.2mm,black] (5,0) -- (6,1);
\draw [line width=0.2mm,black] (5,0) -- (7,1) -- (5,2);
\draw [line width=0.2mm,black] (6,0) -- (7,0.5);

\draw [line width=0.2mm,black] (2,1) -- (2,2) -- (3,3);
\draw [line width=0.2mm,black] (5,1) -- (4,2) -- (4,3);
\draw [line width=0.2mm,black] (7,1) -- (6,2) -- (6,3);
\draw [line width=0.2mm,black] (7,1) -- (7,2) -- (7,3);

\draw [line width=0.2mm,black] (1,2) -- (1,3);
\draw [line width=0.2mm,black] (1,2) -- (2,3) -- (1,4);
\draw [line width=0.2mm,black] (4,2) -- (5,3) -- (5,4);

\draw [line width=0.2mm,black] (2,3) -- (2,4) -- (1,5);
\draw [line width=0.2mm,black] (3,3) -- (3,4);
\draw [line width=0.2mm,black] (3,3) -- (4,4) -- (2,5);
\draw [line width=0.2mm,black] (5,3) -- (6,4);
\draw [line width=0.2mm,black] (6,3) -- (7,4) -- (5,5);

\draw [line width=0.2mm,black] (4,4) -- (3,5);
\draw [line width=0.2mm,black] (5,4) -- (4,5) -- (4,6);
\draw [line width=0.2mm,black] (7,4) -- (6,5); 

\draw [line width=0.2mm,black] (1,5) -- (1,6); 
\draw [line width=0.2mm,black] (1,5) -- (2,6) -- (1,7); 
\draw [line width=0.2mm,black] (1,5) -- (3,6);
\draw [line width=0.2mm,black] (5,5) -- (5,6) -- (3,7); 
\draw [line width=0.2mm,black] (5,5) -- (6,6) -- (5,7);
\draw [line width=0.2mm,black] (7,5) -- (7,6) -- (6,7);

\draw [line width=0.2mm,black] (2,6) -- (2,7);
\draw [line width=0.2mm,black] (5,6) -- (4,7);
\draw [line width=0.2mm,black] (7,6) -- (7,7);

\node[black] () at (-0.5,0){$-3$};
\node[black] () at (-0.5,1){$-2$};
\node[black] () at (-0.5,2){$-1$};
\node[black] () at (-0.3,3){$0$};
\node[black] () at (-0.3,4){$1$};
\node[black] () at (-0.3,5){$2$};
\node[black] () at (-0.3,6){$3$};
\node[black] () at (-0.3,7){$4$};

\node[black] () at (0,-0.4){$0$};
\node[black] () at (1,-0.4){$1$};
\node[black] () at (2,-0.4){$2$};
\node[black] () at (3,-0.4){$3$};
\node[black] () at (4,-0.4){$4$};
\node[black] () at (5,-0.4){$5$};
\node[black] () at (6,-0.4){$6$};
\node[black] () at (7,-0.4){$7$};

\end{tikzpicture}
       \caption{}
    \end{subfigure}
    \begin{subfigure}[t]{0.32\textwidth}
         \begin{tikzpicture}[scale=0.6]

\draw[step=1cm,dotted, thin] (0,0) grid (7,7);

\draw [line width=0.2mm,black] (3,-0.03) -- (3,0.03);
\draw [line width=0.2mm,black] (7,-0.03) -- (7,0.03);

\draw [line width=0.2mm,black] (0,0) -- (0,7);
\draw [line width=0.2mm,black] (1,0) -- (1,1);
\draw [line width=0.2mm,black] (2,0) -- (2,1) -- (1,2);
\draw [line width=0.2mm,black] (2,0) -- (3,1);
\draw [line width=0.2mm,black] (4,0) -- (4,1);
\draw [line width=0.2mm,black] (4,0) -- (5,1) -- (3,2);
\draw [line width=0.2mm,black] (5,0) -- (6,1);
\draw [line width=0.2mm,black] (5,0) -- (7,1) -- (5,2);
\draw [line width=0.2mm,black] (6,0) -- (7,0.5);

\draw [line width=0.2mm,black] (2,1) -- (2,2) -- (3,3);
\draw [line width=0.2mm,black] (5,1) -- (4,2) -- (4,3);
\draw [line width=0.2mm,black] (7,1) -- (6,2) -- (6,3);
\draw [line width=0.2mm,black] (7,1) -- (7,2) -- (7,3);

\draw [line width=0.2mm,black] (1,2) -- (1,3);
\draw [line width=0.2mm,black] (1,2) -- (2,3) -- (1,4);
\draw [line width=0.2mm,black] (4,2) -- (5,3) -- (5,4);

\draw [line width=0.2mm,black] (2,3) -- (2,4) -- (1,5);
\draw [line width=0.2mm,black] (3,3) -- (3,4);
\draw [line width=0.2mm,black] (3,3) -- (4,4) -- (2,5);
\draw [line width=0.2mm,black] (5,3) -- (6,4);
\draw [line width=0.2mm,black] (6,3) -- (7,4) -- (5,5);

\draw [line width=0.2mm,black] (4,4) -- (3,5);
\draw [line width=0.2mm,black] (5,4) -- (4,5) -- (4,6);
\draw [line width=0.2mm,black] (7,4) -- (6,5); 

\draw [line width=0.2mm,black] (1,5) -- (1,6); 
\draw [line width=0.2mm,black] (1,5) -- (2,6) -- (1,7); 
\draw [line width=0.2mm,black] (1,5) -- (3,6);
\draw [line width=0.2mm,black] (5,5) -- (5,6) -- (3,7); 
\draw [line width=0.2mm,black] (5,5) -- (6,6) -- (5,7);
\draw [line width=0.2mm,black] (7,5) -- (7,6) -- (6,7);

\draw [line width=0.2mm,black] (2,6) -- (2,7);
\draw [line width=0.2mm,black] (5,6) -- (4,7);
\draw [line width=0.2mm,black] (7,6) -- (7,7);

\node[black] () at (-0.5,0){$-3$};
\node[black] () at (-0.5,1){$-2$};
\node[black] () at (-0.5,2){$-1$};
\node[black] () at (-0.3,3){$0$};
\node[black] () at (-0.3,4){$1$};
\node[black] () at (-0.3,5){$2$};
\node[black] () at (-0.3,6){$3$};
\node[black] () at (-0.3,7){$4$};

\node[black] () at (0,-0.4){$0$};
\node[black] () at (1,-0.4){$1$};
\node[black] () at (2,-0.4){$2$};
\node[black] () at (3,-0.4){$3$};
\node[black] () at (4,-0.4){$4$};
\node[black] () at (5,-0.4){$5$};
\node[black] () at (6,-0.4){$6$};
\node[black] () at (7,-0.4){$7$};

\draw [line width=0.2mm,red] (0.5,0) -- (0.5,7);

\draw [line width=0.2mm,red] (1.5,7.03) -- (1.5,6.97);
\draw [line width=0.2mm,red] (3.5,7.03) -- (3.5,6.97);
\draw [line width=0.2mm,red] (6.5,7.03) -- (6.5,6.97);

\draw [line width=0.2mm,red]  (5.5,7)-- (6.5,6)-- (6.5,5);
\draw [line width=0.2mm,red]  (4.5,7)-- (5.5,6);
\draw [line width=0.2mm,red] (2.5,7) --(4.5,6)--(4.5,5);
\draw [line width=0.2mm,red]  (2.5,7)-- (3.5,6)-- (3.5,5);
\draw [line width=0.2mm,red]  (2.5,7)-- (2.5,6);
\draw [line width=0.2mm,red]  (0.5,7)-- (1.5,6);

\draw [line width=0.2mm,red]  (6.5,6)-- (5.5,5);
\draw [line width=0.2mm,red]  (3.5,6)-- (2.5,5);
\draw [line width=0.2mm,red]  (3.5,6)-- (1.5,5)-- (3.5,4);

\draw [line width=0.2mm,red]  (6.5,5)-- (7,4.5);
\draw [line width=0.2mm,red]  (4.5,5)-- (6.5,4)-- (5.5,3);
\draw [line width=0.2mm,red]  (4.5,5)-- (5.5,4);
\draw [line width=0.2mm,red] (3.5,5)--(4.5,4)--(4.5,3);
\draw [line width=0.2mm,red]  (1.5,5)-- (2.5,4)-- (2.5,3);
\draw [line width=0.2mm,red]  (0.5,5)-- (1.5,4);

\draw [line width=0.2mm,red]  (7,3.5)-- (6.5,3)-- (6.5,2);
\draw [line width=0.2mm,red]  (4.5,4)-- (3.5,3)-- (3.5,2);
\draw [line width=0.2mm,red]  (0.5,4)--(1.5,3);

\draw [line width=0.2mm,red]  (5.5,3)-- (5.5,2);
\draw [line width=0.2mm,red]  (5.5,3)-- (4.5,2);
\draw [line width=0.2mm,red]  (3.5,3)-- (2.5,2)-- (3.5,1);
\draw [line width=0.2mm,red]  (2.5,3)-- (1.5,2);

\draw [line width=0.2mm,red]  (2.5,2)-- (4.5,1);
\draw [line width=0.2mm,red]  (4.5,2)-- (5.5,1)-- (4.5,0);
\draw [line width=0.2mm,red]  (4.5,2)-- (6.5,1);
\draw [line width=0.2mm,red]  (2.5,2)--(2.5,1);
\draw [line width=0.2mm,red]  (0.5,2)-- (1.5,1)--(1.5,0);

\draw [line width=0.2mm,red] (7,0.25) -- (6.5,0);
\draw [line width=0.2mm,red]  (7,0.75)-- (5.5,0);
\draw [line width=0.2mm,red]  (3.5,1)-- (3.5,0);
\draw [line width=0.2mm,red]  (3.5,1)-- (2.5,0);

\end{tikzpicture}
       \caption{}
    \end{subfigure}
    \caption{Graphical representation of the iterated reproduction mappings. }
\label{Fig1}\end{figure*}
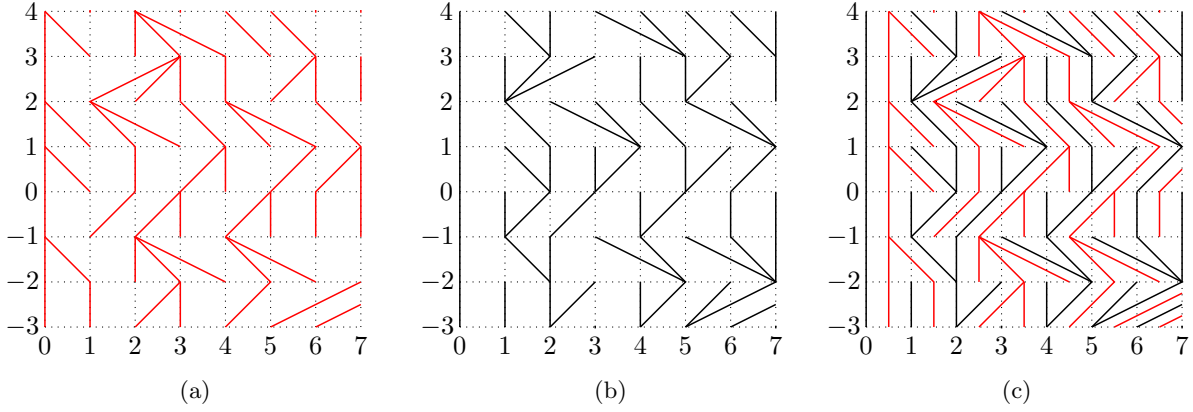

\section{Coupled rank-dependent GW processes and their duals}\label{se}
Let  $ \Phi_0$  be  the class of monotone functions $U:\mathbb N_0\to \mathbb N_0$ such that $U(0)=0$. If  $U\in \Phi_0$ and  $u(x)=U(x)-U(x-1)$, then $U$ will be called a {\it reproduction mapping} with the offspring numbers $u(x)$, $x\in\mathbb N$.
Given a set of independent  random variables \eqref{ut}, define a sequence of random reproduction mappings $U_t(x)=\sum_{y=1}^x u_t(y)$, and consider the family of stochastic iterations 
\begin{equation}\label{zt}
U_{a,b}=U_{b-1}\circ U_{b-2}\circ\cdots\circ U_{a}, \quad a<b,\quad U_{a,a}(x)\equiv x.
\end{equation}
Putting  $Z_t=U_{a,t}(z)$, we obtain a time-inhomogeneous Markov chain  $\{Z_t\}_{t\ge a}$ satisfying \eqref{Yt}, which will be called a rank-dependent GW process. Moreover, using the system of stochastic iterations 
\begin{equation}\label{bU}
\boldsymbol U=\{U_{a,b}\}_{\infty<a\le b<\infty},
\end{equation}
we can define  coupled Markov chains $\{U_{a,t}(x)\}_{t\ge a}$  starting at $U_{a,a}(x)=x$ for all possible $a\in\mathbb Z$ and $x\in\mathbb N$. We call \eqref{bU} a rank-dependent GW system with the reproduction law
$$f_{t,x}(s)=\rE s^{u_{t}(x)},\quad  s\in[0,1],\quad  x\in\mathbb N,\quad  t\in\mathbb Z.$$

\begin{definition}\label{deg}
If $U\in \Phi_0$ and $V=U^-$, where
\[ U^-(x)=\min\{y:U(y)\ge x\},
 \]
then $V\in \Phi_0$ will be called the  {\it pathwise dual} of the reproduction mapping $U$.
\end{definition}
As shown in Section \ref{Spro}, Definition \ref{deg} is equivalent to the equality
\begin{align}
 \{(x,y)\in\mathbb N_0^2:V(x)\le y\}= \{(x,y)\in\mathbb N_0^2:U(y)\ge x\},\label{VxU}
\end{align}
and therefore can be referred to as the pathwise Siegmund duality, see \cite{JK,Si,StS}. 
\begin{definition}\label{dog}
Given a rank-dependent GW system \eqref{bU}, define its  time-reverse by
\begin{equation*}
\boldsymbol V=\{V_{b,a}\}_{\infty<a\le b<\infty},\quad V_{b,a}=V_{a}\circ\cdots\circ V_{b-1}, \quad  a\le b,
\end{equation*}
where $V_t=U_t^-$ are the dual reproduction mappings. Putting $\hat U_t=V_{-t-1}$, define the pathwise dual of $\boldsymbol U$ by
\begin{equation*}
\hat{\boldsymbol U}=\{\hat U_{a,b}\}_{\infty<a\le b<\infty},\quad \hat U_{a,b}=\hat U_{b-1}\circ\cdots\circ \hat U_{a}, \quad  a\le b,
\end{equation*}
\end{definition}
The trajectories of a rank-dependent GW system and its time-reverse can be represented by forest graphs on the grid of nodes $\mathbb N_0\times \mathbb Z$. As seen on the Figure \ref{Fig1}a, the bottom-up lineages $\{(U_{a,t}(x),t), t\ge a\}
$  starting from different levels $a\in\mathbb Z$ and different positions  $x\in\mathbb N_0$, merge together into coalescent trees. The resulting graph will be called a {\it dual forest}. On the other hand, as shown on the Figure \ref{Fig1}b, the top-down  lineages $\{(V_{b,t}(x),t), t\le b\}$
starting from different levels $b\in\mathbb Z$ and different positions  $x\in\mathbb N_0$, build up a graph that we call a {\it primary forest}. 
Figure \ref{Fig1}c demonstrates that the two forests can be  conveniently depicted together after the dual forest is shifted to the right by $\sfrac{1}{2}$.
Drawn in this way, the lineages of the primary and dual forests do not cross. The primary forest describes the genealogical  trees of the primary rank-dependent GW system. A lineage in the dual tree followed up from the vertex $(z,t)$, delineates a trajectory of the Markov chain \eqref{Yt}. Figure \ref{Fig2} illustrates how the trajectories of $\hat{\boldsymbol U}$ are obtained from those of $\boldsymbol V$ by a vertical flipping.

\section{Main results}\label{Sdu}
Let reproduction mappings $U,V,\tilde U\in\Phi_0$, 
be connected as  $V=U^-$,  $\tilde U=V^-$. The same notation will be also used with the time index $t$.

\begin{theorem}\label{proM}
Given  \eqref{bU}, put 
$$V_t=U_t^-,\quad \hat U_t=V_{-t-1},\quad \hat V_t=\hat U_t^-,\quad \tilde  U_t=\hat V_{-t-1},$$
and alongside the rank-dependent GW system $\boldsymbol U$, consider its time-reverse  $\boldsymbol V$ and pathwise dual $\hat {\boldsymbol U}$, see Definition \ref{dog}.
Whenever  $a\le b$ and $x,y\in\mathbb N_0$,  the two events coincide
 \[\{\hat U_{-b,-a}(x)\le y\}=\{x\le U_{a,b}(y)\}.\]
The reproduction mappings $\tilde  U_t$ define a rank-dependent GW system  $\tilde {\boldsymbol U}$ in a similar way as \eqref{bU} defines the primary rank-dependent GW system  $\boldsymbol U$. The $\tilde {\boldsymbol U}$ is the dual of the dual $\hat {\boldsymbol U}$, and is obtained as a simple shifting transform of  ${\boldsymbol U}$:
\[\tilde U_{a,b}(x)= U_{a,b}(x-1)+1,\quad x\ge1,\quad a\le b.\]
\end{theorem}

Figure \ref{Fig2} is an illustration of the "picture proof" of Theorem \ref{proM} using the graphical representation.  Figure \ref{Fig2}b shows an intermediate step in the transformation of the primary forest on Figure \ref{Fig2}a into  the primary forest on Figure \ref{Fig2}c representing the twofold dual rank-dependent GW system $\tilde{\boldsymbol U}$. Since $\hat U_{a,b}=V_{-b,-a}$, we find that the primary forest on  Figure \ref{Fig2}b representing the genealogical trees of the dual rank-dependent GW system $\hat{\boldsymbol U}$, is the dual forest from Figure \ref{Fig2}a flipped around the axis $t=0$ and shifted to the right by 1 (visually it is shifted just by $\sfrac{1}{2}$). Observe that the rank 1  particle in the dual reproduction flow is necessarily "eternal", giving birth at least to one offspring.  Thus the black forest from (a) flips into the red forest in (b), then the black forest in (b) generates the red forest in (b), which in turn gives the black forest in (c).
 We see that the primary forest on
Figure \ref{Fig2}c is a shifted copy of the primary forest from Figure \ref{Fig2}a, as predicted by Theorem \ref{proM}.

\begin{figure*}[t]
    \centering
    \begin{subfigure}[t]{0.32\textwidth}
         \begin{tikzpicture}[scale=0.6]

\draw[step=1cm,dotted, thin] (0,0) grid (7,7);
\draw [line width=0.2mm,black] (0,0) -- (0,7);

\draw [line width=0.2mm,black] (3,-0.03) -- (3,0.03);
\draw [line width=0.2mm,black] (7,-0.03) -- (7,0.03);

\draw [line width=0.2mm,black] (1,0) -- (1,1);
\draw [line width=0.2mm,black] (2,0) -- (2,1) -- (1,2);
\draw [line width=0.2mm,black] (2,0) -- (3,1);
\draw [line width=0.2mm,black] (4,0) -- (4,1);
\draw [line width=0.2mm,black] (4,0) -- (5,1) -- (3,2);
\draw [line width=0.2mm,black] (5,0) -- (6,1);
\draw [line width=0.2mm,black] (5,0) -- (7,1) -- (5,2);
\draw [line width=0.2mm,black] (6,0) -- (7,0.5);

\draw [line width=0.2mm,black] (2,1) -- (2,2) -- (3,3);
\draw [line width=0.2mm,black] (5,1) -- (4,2) -- (4,3);
\draw [line width=0.2mm,black] (7,1) -- (6,2) -- (6,3);
\draw [line width=0.2mm,black] (7,1) -- (7,2) -- (7,3);

\draw [line width=0.2mm,black] (1,2) -- (1,3);
\draw [line width=0.2mm,black] (1,2) -- (2,3) -- (1,4);
\draw [line width=0.2mm,black] (4,2) -- (5,3) -- (5,4);

\draw [line width=0.2mm,black] (2,3) -- (2,4) -- (1,5);
\draw [line width=0.2mm,black] (3,3) -- (3,4);
\draw [line width=0.2mm,black] (3,3) -- (4,4) -- (2,5);
\draw [line width=0.2mm,black] (5,3) -- (6,4);
\draw [line width=0.2mm,black] (6,3) -- (7,4) -- (5,5);

\draw [line width=0.2mm,black] (4,4) -- (3,5);
\draw [line width=0.2mm,black] (5,4) -- (4,5) -- (4,6);
\draw [line width=0.2mm,black] (7,4) -- (6,5); 

\draw [line width=0.2mm,black] (1,5) -- (1,6); 
\draw [line width=0.2mm,black] (1,5) -- (2,6) -- (1,7); 
\draw [line width=0.2mm,black] (1,5) -- (3,6);
\draw [line width=0.2mm,black] (5,5) -- (5,6) -- (3,7); 
\draw [line width=0.2mm,black] (5,5) -- (6,6) -- (5,7);
\draw [line width=0.2mm,black] (7,5) -- (7,6) -- (6,7);

\draw [line width=0.2mm,black] (2,6) -- (2,7);
\draw [line width=0.2mm,black] (5,6) -- (4,7);
\draw [line width=0.2mm,black] (7,6) -- (7,7);

\node[black] () at (-0.5,0){$-3$};
\node[black] () at (-0.5,1){$-2$};
\node[black] () at (-0.5,2){$-1$};
\node[black] () at (-0.3,3){$0$};
\node[black] () at (-0.3,4){$1$};
\node[black] () at (-0.3,5){$2$};
\node[black] () at (-0.3,6){$3$};
\node[black] () at (-0.3,7){$4$};

\node[black] () at (1,-0.4){$1$};
\node[black] () at (2,-0.4){$2$};
\node[black] () at (3,-0.4){$3$};
\node[black] () at (4,-0.4){$4$};
\node[black] () at (5,-0.4){$5$};
\node[black] () at (6,-0.4){$6$};
\node[black] () at (7,-0.4){$7$};

\draw [line width=0.2mm,red] (0.5,0) -- (0.5,7);

\draw [line width=0.2mm,red] (1.5,7.03) -- (1.5,6.97);
\draw [line width=0.2mm,red] (2.5,7.03) -- (2.5,6.97);
\draw [line width=0.2mm,red] (3.5,7.03) -- (3.5,6.97);
\draw [line width=0.2mm,red] (6.5,7.03) -- (6.5,6.97);

\draw [line width=0.2mm,red]  (5.5,7)-- (6.5,6)-- (6.5,5);
\draw [line width=0.2mm,red]  (4.5,7)-- (5.5,6);
\draw [line width=0.2mm,red] (2.5,7) --(4.5,6)--(4.5,5);
\draw [line width=0.2mm,red]  (2.5,7)-- (3.5,6)-- (3.5,5);
\draw [line width=0.2mm,red]  (2.5,7)-- (2.5,6);
\draw [line width=0.2mm,red]  (0.5,7)-- (1.5,6);

\draw [line width=0.2mm,red]  (6.5,6)-- (5.5,5);
\draw [line width=0.2mm,red]  (3.5,6)-- (2.5,5);
\draw [line width=0.2mm,red]  (3.5,6)-- (1.5,5)-- (3.5,4);

\draw [line width=0.2mm,red]  (6.5,5)-- (7,4.5);
\draw [line width=0.2mm,red]  (4.5,5)-- (6.5,4)-- (5.5,3);
\draw [line width=0.2mm,red]  (4.5,5)-- (5.5,4);
\draw [line width=0.2mm,red] (3.5,5)--(4.5,4)--(4.5,3);
\draw [line width=0.2mm,red]  (1.5,5)-- (2.5,4)-- (2.5,3);
\draw [line width=0.2mm,red]  (0.5,5)-- (1.5,4);

\draw [line width=0.2mm,red]  (7,3.5)-- (6.5,3)-- (6.5,2);
\draw [line width=0.2mm,red]  (4.5,4)-- (3.5,3)-- (3.5,2);
\draw [line width=0.2mm,red]  (0.5,4)--(1.5,3);

\draw [line width=0.2mm,red]  (5.5,3)-- (5.5,2);
\draw [line width=0.2mm,red]  (5.5,3)-- (4.5,2);
\draw [line width=0.2mm,red]  (3.5,3)-- (2.5,2)-- (3.5,1);
\draw [line width=0.2mm,red]  (2.5,3)-- (1.5,2);

\draw [line width=0.2mm,red]  (2.5,2)-- (4.5,1);
\draw [line width=0.2mm,red]  (4.5,2)-- (5.5,1)-- (4.5,0);
\draw [line width=0.2mm,red]  (4.5,2)-- (6.5,1);
\draw [line width=0.2mm,red]  (2.5,2)--(2.5,1);
\draw [line width=0.2mm,red]  (0.5,2)-- (1.5,1)--(1.5,0);

\draw [line width=0.2mm,red] (7,0.25) -- (6.5,0);
\draw [line width=0.2mm,red]  (7,0.75)-- (5.5,0);
\draw [line width=0.2mm,red]  (3.5,1)-- (3.5,0);
\draw [line width=0.2mm,red]  (3.5,1)-- (2.5,0);

\end{tikzpicture}
   
        \caption{}
    \end{subfigure}%
    ~ 
    \begin{subfigure}[t]{0.32\textwidth}
         \begin{tikzpicture}[scale=0.6]

\draw[step=1cm,dotted, thin] (0,0) grid (7,7);
\draw [line width=0.2mm,black] (0,0) -- (0,7);

\draw [line width=0.2mm,black] (2,-0.03) -- (2,0.03);
\draw [line width=0.2mm,black] (7,-0.03) -- (7,0.03);

\draw [line width=0.2mm,black] (1,0) -- (1,7);
\draw [line width=0.2mm,black] (1,0) -- (2,1);
\draw [line width=0.2mm,black] (3,0) -- (3,1);
\draw [line width=0.2mm,black] (3,0) -- (4,1)--(2,2)--(3,3)--(3,4)--(2,5);
\draw [line width=0.2mm,black] (3,0) -- (5,1) -- (5,2)-- (6,3);
\draw [line width=0.2mm,black] (5,0) -- (6,1);
\draw [line width=0.2mm,black] (6,0) -- (7,1)--(6,2);
\draw [line width=0.2mm,black] (4,1) -- (3,2);
\draw [line width=0.2mm,black] (4,1) -- (4,2) -- (5,3)--(4,4)--(3,5)--(3,6);
\draw [line width=0.2mm,black] (7,1) -- (7,2);
\draw [line width=0.2mm,black] (1,2) -- (2,3);
\draw [line width=0.2mm,black] (2,2) -- (4,3);
\draw [line width=0.2mm,black] (5,2) -- (7,3) -- (6,4) -- (5,5) -- (6,6) -- (5,7);
\draw [line width=0.2mm,black] (1,3) -- (2,4);
\draw [line width=0.2mm,black] (5,3) -- (5,4);
%
\draw [line width=0.2mm,black] (4,4) -- (4,5);
\draw [line width=0.2mm,black] (6,4) -- (6,5);
\draw [line width=0.2mm,black] (7,4) -- (7,5);
%
\draw [line width=0.2mm,black] (1,5) -- (2,6)-- (2,7); 
\draw [line width=0.2mm,black] (3,5) -- (4,6) -- (3,7);
\draw [line width=0.2mm,black] (3,5) -- (5,6); 
\draw [line width=0.2mm,black] (5,5) -- (7,6);
%
\draw [line width=0.2mm,black] (4,6) -- (4,7);
\draw [line width=0.2mm,black] (7,6.5) -- (6,7);

\node[black] () at (-0.5,0){$-4$};
\node[black] () at (-0.5,1){$-3$};
\node[black] () at (-0.5,2){$-2$};
\node[black] () at (-0.5,3){$-1$};
\node[black] () at (-0.3,4){$0$};
\node[black] () at (-0.3,5){$1$};
\node[black] () at (-0.3,6){$2$};
\node[black] () at (-0.3,7){$3$};

\node[black] () at (1,-0.4){$1$};
\node[black] () at (2,-0.4){$2$};
\node[black] () at (3,-0.4){$3$};
\node[black] () at (4,-0.4){$4$};
\node[black] () at (5,-0.4){$5$};
\node[black] () at (6,-0.4){$6$};
\node[black] () at (7,-0.4){$7$};

%
\draw [line width=0.2mm,red] (0.5,0) -- (0.5,7);
%
\draw [line width=0.2mm,red] (3.5,7.03) -- (3.5,6.97);
%
\draw [line width=0.2mm,red]  (4.5,7)-- (5.5,6)-- (3.5,5);
\draw [line width=0.2mm,red]  (4.5,7)-- (4.5,6);
\draw [line width=0.2mm,red] (5.5,7) --(6.5,6);
\draw [line width=0.2mm,red] (5.5,7) --(7,6.25);
\draw [line width=0.2mm,red]  (2.5,7)-- (2.5,6)-- (1.5,5)-- (1.5,4);
\draw [line width=0.2mm,red]  (2.5,7)-- (3.5,6);
\draw [line width=0.2mm,red] (1.5,7) -- (1.5,6);
\draw [line width=0.2mm,red]  (6.5,7)-- (7,6.75);
\draw [line width=0.2mm,red]  (2.5,6)-- (2.5,5)-- (3.5,4)-- (3.5,3);
\draw [line width=0.2mm,red]  (7,5.5)-- (6.5,5);
\draw [line width=0.2mm,red]  (7,5.75)-- (5.5,5);
\draw [line width=0.2mm,red]  (5.5,6)-- (4.5,5)-- (4.5,4);
\draw [line width=0.2mm,red]  (1.5,5)--(2.5,4)--(1.5,3);
\draw [line width=0.2mm,red]  (4.5,5)-- (5.5,4)-- (5.5,3)-- (4.5,2)-- (4.5,1);
\draw [line width=0.2mm,red] (6.5,5)--(6.5,4)--(7,3.5);
%
%
\draw [line width=0.2mm,red]  (4.5,3)-- (3.5,2);
\draw [line width=0.2mm,red]  (7,2.75)-- (5.5,2);
\draw [line width=0.2mm,red]  (7,2.5)-- (6.5,2);
\draw [line width=0.2mm,red]  (2.5,4)-- (2.5,3)-- (1.5,2)-- (1.5,1);
\draw [line width=0.2mm,red]  (3.5,4)-- (4.5,3)-- (2.5,2);
\draw [line width=0.2mm,red]  (5.5,4)-- (6.5,3);
\draw [line width=0.2mm,red]  (1.5,2)-- (3.5,1);
\draw [line width=0.2mm,red]  (1.5,2)-- (2.5,1)-- (1.5,0);
\draw [line width=0.2mm,red]  (5.5,2)-- (5.5,1)-- (3.5,0);
\draw [line width=0.2mm,red]  (5.5,2)-- (6.5,1)-- (5.5,0);
\draw [line width=0.2mm,red]  (2.5,1)--(2.5,0);
\draw [line width=0.2mm,red]  (5.5,1)--(4.5,0);
\draw [line width=0.2mm,red] (7,0.5) -- (6.5,0);
%

\end{tikzpicture}
       \caption{}
    \end{subfigure}
    \begin{subfigure}[t]{0.32\textwidth}
         \begin{tikzpicture}[scale=0.6]

\draw[step=1cm,dotted, thin] (0,0) grid (7,7);
\draw [line width=0.2mm,black] (1,0) -- (1,7);
\draw [line width=0.2mm,black] (0,0) -- (0,7);

\draw [line width=0.2mm,black] (3,-0.03) -- (3,0.03);

\draw [line width=0.2mm,black] (2,0) -- (2,1);
\draw [line width=0.2mm,black] (3,0) -- (3,1) -- (2,2);
\draw [line width=0.2mm,black] (3,0) -- (4,1);
\draw [line width=0.2mm,black] (5,0) -- (5,1);
\draw [line width=0.2mm,black] (5,0) -- (6,1) -- (4,2);
\draw [line width=0.2mm,black] (6,0) -- (7,1);
\draw [line width=0.2mm,black] (6,0) -- (7,0.5);

\draw [line width=0.2mm,black] (3,1) -- (3,2) -- (4,3);
\draw [line width=0.2mm,black] (6,1) -- (5,2) -- (5,3);
\draw [line width=0.2mm,black] (7,2) -- (7,3);
\draw [line width=0.2mm,black] ((7,1.5)--(6,2);

\draw [line width=0.2mm,black] (2,2) -- (2,3);
\draw [line width=0.2mm,black] (2,2) -- (3,3) -- (2,4);
\draw [line width=0.2mm,black] (5,2) -- (6,3) -- (6,4);

\draw [line width=0.2mm,black] (3,3) -- (3,4) -- (2,5);
\draw [line width=0.2mm,black] (4,3) -- (4,4);
\draw [line width=0.2mm,black] (4,3) -- (5,4) -- (3,5);
\draw [line width=0.2mm,black] (6,3) -- (7,4);
\draw [line width=0.2mm,black] (7,4.5) -- (6,5) -- (6,6);

\draw [line width=0.2mm,black] (5,4) -- (4,5);
\draw [line width=0.2mm,black] (6,4) -- (5,5)--(5,6);

\draw [line width=0.2mm,black] (2,5) -- (2,6); 
\draw [line width=0.2mm,black] (2,5) -- (3,6) -- (2,7); 
\draw [line width=0.2mm,black] (2,5) -- (4,6);
\draw [line width=0.2mm,black] (6,5) -- (6,6) -- (4,7); 
\draw [line width=0.2mm,black] (6,5) -- (7,6) -- (6,7);

\draw [line width=0.2mm,black] (3,6) -- (3,7);
\draw [line width=0.2mm,black] (6,6) -- (5,7);

\node[black] () at (-0.5,0){$-3$};
\node[black] () at (-0.5,1){$-2$};
\node[black] () at (-0.5,2){$-1$};
\node[black] () at (-0.3,3){$0$};
\node[black] () at (-0.3,4){$1$};
\node[black] () at (-0.3,5){$2$};
\node[black] () at (-0.3,6){$3$};
\node[black] () at (-0.3,7){$4$};

\node[black] () at (1,-0.4){$1$};
\node[black] () at (2,-0.4){$2$};
\node[black] () at (3,-0.4){$3$};
\node[black] () at (4,-0.4){$4$};
\node[black] () at (5,-0.4){$5$};
\node[black] () at (6,-0.4){$6$};
\node[black] () at (7,-0.4){$7$};

\draw [line width=0.2mm,red] (0.5,0) -- (0.5,7);
\draw [line width=0.2mm,red] (1.5,0) -- (1.5,7);

\draw [line width=0.2mm,red] (2.5,7.03) -- (2.5,6.97);
\draw [line width=0.2mm,red] (3.5,7.03) -- (3.5,6.97);
\draw [line width=0.2mm,red] (4.5,7.03) -- (4.5,6.97);

\draw [line width=0.2mm,red]  (6.5,7)-- (7,6.5);
\draw [line width=0.2mm,red]  (5.5,7)-- (6.5,6);
\draw [line width=0.2mm,red] (3.5,7) --(5.5,6)--(5.5,5);
\draw [line width=0.2mm,red]  (3.5,7)-- (4.5,6)-- (4.5,5);
\draw [line width=0.2mm,red]  (3.5,7)-- (3.5,6);
\draw [line width=0.2mm,red]  (1.5,7)-- (2.5,6);

\draw [line width=0.2mm,red]  (7,5.5)-- (6.5,5);
\draw [line width=0.2mm,red]  (4.5,6)-- (3.5,5);
\draw [line width=0.2mm,red]  (4.5,6)-- (2.5,5)-- (4.5,4);

\draw [line width=0.2mm,red]  (5.5,5)-- (7,4.25);
\draw [line width=0.2mm,red]  (7,3.5)-- (6.5,3);
\draw [line width=0.2mm,red]  (5.5,5)-- (6.5,4);
\draw [line width=0.2mm,red] (4.5,5)--(5.5,4)--(5.5,3);
\draw [line width=0.2mm,red]  (2.5,5)-- (3.5,4)-- (3.5,3);
\draw [line width=0.2mm,red]  (1.5,5)-- (2.5,4);

\draw [line width=0.2mm,red]  (5.5,4)-- (4.5,3)-- (4.5,2);
\draw [line width=0.2mm,red]  (1.5,4)--(2.5,3);

\draw [line width=0.2mm,red]  (6.5,3)-- (6.5,2);
\draw [line width=0.2mm,red]  (6.5,3)-- (5.5,2);
\draw [line width=0.2mm,red]  (4.5,3)-- (3.5,2)-- (4.5,1);
\draw [line width=0.2mm,red]  (3.5,3)-- (2.5,2);

\draw [line width=0.2mm,red]  (3.5,2)-- (5.5,1);
\draw [line width=0.2mm,red]  (5.5,2)-- (6.5,1)-- (5.5,0);
\draw [line width=0.2mm,red]  (5.5,2)-- (7,1.25);
\draw [line width=0.2mm,red]  (3.5,2)--(3.5,1);
\draw [line width=0.2mm,red]  (1.5,2)-- (2.5,1)--(2.5,0);

\draw [line width=0.2mm,red]  (7,0.25)-- (6.5,0);
\draw [line width=0.2mm,red]  (4.5,1)-- (4.5,0);
\draw [line width=0.2mm,red]  (4.5,1)-- (3.5,0);

\end{tikzpicture}
       \caption{}
    \end{subfigure}
    \caption{Graphical illustration of Theorem \ref{proM}. The twofold dual on panel (c) is a shifted copy of the primary forest given in (a).}
\label{Fig2}\end{figure*}
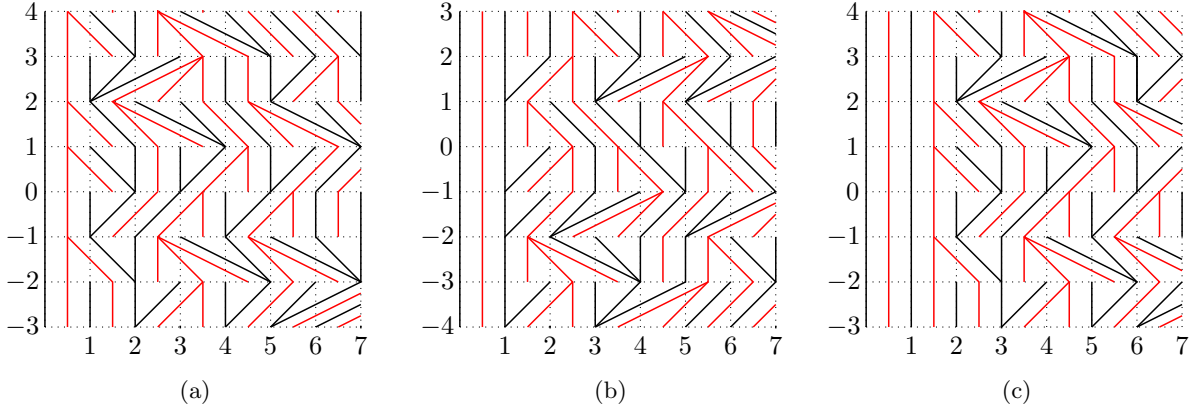

The next result describes the case of a GW  reproduction mapping $U$, that is when the corresponding offspring numbers $\{u(x),x\in\mathbb N\}$ are independent and identically distributed.
\begin{proposition}\label{JJ}
Referring to Definition \ref{deg} consider a reproduction mapping $U$ together with its dual mapping $V$. 
The  mapping $U$ generates a GW reproduction law
$$\rP(u(x)=k)=P_{k},\quad k\ge0,\quad x\ge1,$$ 
if and only if the dual  offspring numbers  have a representation
\begin{align}
  (v(1),v(2),\ldots)&=(\xi_{1}+1,\underbrace{0,\ldots,0}_{\eta_1},\xi_{2}+1,\underbrace{0,\ldots,0}_{\eta_2},\ldots), \label{hata}
 \end{align}
where $\xi_1,\eta_1,\xi_2,\eta_2,\ldots$ are mutually independent $\mathbb N_0$-valued random variables with the marginal distributions
\begin{align*}
 \rP( \xi_i=k)=P_{0}^k(1-P_{0}),\quad
 \rP(\eta_i=k)={P_{k+1}\over 1-P_{0}},\quad k\ge0,\quad i\ge1.
\end{align*}
In this case, the marginal dual reproduction law has a linear-fractional distribution
\begin{align}
\rE s^{v(x)}&=1-\hat q(x)+\hat q(x){P_{0}s\over 1-(1-P_{0})s},  \label{dlf}
\end{align}
where 
\begin{align}
\hat q(1)=1,\quad \hat q(x)=(1-P_{0})^{-1}\sum _{k=1}^{x-1}P_{k}\hat q(x-k),\quad x\ge2.\label{p0x}
\end{align}

\end{proposition}

 A natural question arising in connection to Proposition \ref{JJ} is whether it is possible that the primary and its dual reproduction mappings are both GW?  The answer is no, since the dual law always assigns at least one offspring to the particle of rank 1. The closest the  GW dual one can get is a {\it GW  reproduction with an eternal particle}, which by definition is 
a rank-dependent GW reproduction mapping $V$  whose offspring numbers are such that $v(1)\ge1$, and $v(2),v(3),\ldots$ have a common distribution.
 The following result significantly extends Proposition 3.6 in \cite{KRS}.

\begin{theorem} \label{th}
Consider a GW reproduction mapping $U$. Its pathwise dual $V$ is a GW  reproduction with an eternal particle  if and only if
\begin{align}
 \rE s^{u(x)}&=1-q+q{ps\over 1-(1-p)s},\quad p,q\in(0,1].\label{lft}\end{align}
 In this case,
 \begin{align}
\rE s^{v(1)}&={qs\over 1-(1-q)s}, \qquad \rE s^{v(x)}=1-p+p{qs\over 1-(1-q)s},\quad x\ge 2.  \label{lfg}
\end{align}
 \end{theorem}

\section{Examples 
}\label{BD}

\noindent
{\bf Pure death rank-dependent GW.} 
A distinct forest structure appears in the case when the offspring numbers take values 0 or 1
$$f_{t,x}(s)=p_{t,x}+(1-p_{t,x})s,\quad x\in\mathbb N,\quad t\in\mathbb Z.$$
Each dual lineage followed upwards, eventually vanishes without branching. 
Given $p_{t,x}\equiv p_x$, the dual reproduction  is not rank-dependent GW because of the dependence in the joint distribution
\[\rP(v(1)=k_1,\ldots,v(m)=k_m)=p_1\cdots p_{k_1-1}(1-p_{k_1})\prod_{l=1}^{m-1}p_{k_1+\ldots+k_l+1}\cdots p_{k_1+\ldots+k_{l+1}-1}(1-p_{k_1+\ldots+k_{l+1}}).\]

\noindent
{\bf Birth-death  GW reproduction.} 
Consider a GW reproduction law $\rP(u=k)=p_k$ assuming $p_0+p_1+p_2=1$. If $p_2=0$, then the dual reproduction is GW with a shifted geometric distribution
\[\rP(v(x)=k)=p_0^{k-1}(1-p_0),\quad k\ge1.\]
 If $p_1=0$, then the dual reproduction is rank-dependent GW described by the example given in the Introduction, with $p=p_0$. On the other hand, if $p_0=0$, then the dual reproduction law is not rank-dependent GW because of the following dependence:
\begin{align*}
& \rP(v(1)=1)=1,\quad  \rP(v(2)=0,v(3)=1)=p_2,\\
& \rP(v(2)=1, v(3)=0)=p_1p_2,\quad  \rP(v(1)=0,v(3)=1)=p_1^2.
 \end{align*}


\noindent
{\bf Bounded GW processes.} 
Consider a specific example of the rank-dependent GW process given by
\[f_{t,x}(s)=
\left\{
\begin{array}{rl}
 f(s), &  x\in[1,B_t],  \\
1,  &  x>B_t,    
\end{array}
\right. \quad s\in[0,1], \quad t\ge0.
\]
What we get is a version of a truncated GW process with a stationary reproduction $f$, where the number of particles, allowed to reproduce at time $t$, is bounded by $B_t$. An interesting result for such processes dealing with the supercritical case was obtained in \cite{Z}.

\vspace{0.4cm}
\noindent
{\bf GW processes with immigration.} 
Consider a rank-dependent GW process with $f_{t,1}(s)= sg_t(s)$ and $f_{t,x}\equiv f_t$, $x\ge2$. This is a GW process with an eternal particle  in a varying environment. Removing the eternal particle  of rank 1 and keeping its offspring as immigrants, we arrive at a GW with immigration. The GW process with immigration are well-studied in the case of a stationary reproduction  $f_{t}=f$ and varying immigration $\{g_t\}_{t\ge0}$, see \cite{Ra}.
The case of varying $\{g_t,f_t\}_{t\ge0}$, has got less attention in the literature, see however \cite{MO}.

\vspace{0.4cm}\noindent
{\bf GW processes with emigration.} 
Consider a time-homogeneous GW process with an eternal particle,  such that \eqref{lft} holds for $x\ge2$, and $u(1)\ge1$ has an arbitrary distribution. Its dual Markov chain can be interpreted in terms of a GW process with emigration (catastrophes, disasters), with a random number $\eta_t\stackrel{d}{=}u(1)-1$ of particles being removed from generation $t$. If the current size $Y_t$ does not exceed $\eta_t$, the population dies out. One of the first papers addressing this model was \cite{V}, where the critical case was studied under the assumption $\eta_t\equiv1$. 

It was shown in \cite{Gr} that if the GW reproduction is supercritical and the numbers of emigrants $\{\eta_t\}_{t\ge0}$ are independent copies of $\eta$, then the the GW process with emigration goes extinct with probability 1 iff $\rE \log (\eta+1)=\infty $. On the other hand, a well-known result by \cite{AH} says that a subcritical GW process with immigration has a stationary distribution iff $\rE  \log (\eta+1)<\infty $ for the number of immigrants $u(1)-1\stackrel{d}{=}\eta$.

\vspace{0.4cm}\noindent
{\bf Rank-dependent GW process with a carrying capacity.}  
Consider the time-homogeneous case, $f_{t,x}=f_x$, when the reproduction law is variable along the spatial position. Our setting is suitable for modeling population size dependent reproduction in a way which is different from that of \cite{K,KS}. Let $m_x=f'_x(1)$ be the mean offspring number for the particle of rank $x$. Suppose $m_1>1$ and $m_x$ monotonely decreases with $x$ so that for some $K\in\mathbb N$,
\[m_1+\ldots+m_x\ge x,\quad x\le K,\quad m_1+\ldots+m_x<x,\quad x> K.\]
Such a $K$ can be viewed as the carrying capacity of a population of individuals which produce less than 1 child per individual when the size of the population exceeds $K$.

\vspace{0.4cm}\noindent
{\bf Embeddable rank-dependent GW-processes.} 
Embeddability into continuous time Markov branching processes is not fully resolved issue for basic GW processes \cite[Ch III.12]{AN}. Several explicit examples of embeddable GW processes can be found in \cite{SL}.  One known class of embeddable GW processes in varying environments is the case of linear-fractional reproduction addressed in Theorem \ref{th}. 

Consider a continuous time linear birth-death process $\{Z(t), t\ge t_0\}$ with the variable birth and death rates $\{\lambda(t),\mu(t)\}_{t\in\mathbb R}$ per individual. 
%
%
It is well-known that such a process has linear-fractional distributions. 
By \cite{Ke},
$$\rE(s^{Z(t)}|Z(t_0 )=1)=1-q(t_0,t)+q(t_0,t){p(t_0,t)s\over 1-(1-p(t_0,t))s},$$
where
\begin{align*}
 q(t_0,t)&=\Big(1+\int_{t_0 }^te^{\rho(t_0,u)}\mu(u)du\Big)^{-1},\quad p(t_0,t)=e^{\rho(t_0,t)}q(t_0,t),\quad \rho(t_0,t)=\int_{t_0 }^t(\mu(u)-\lambda(u))du.
\end{align*}
A linear-fractional  GW process with varying parameters $(q_t,p_t)$ in the expression \eqref{lft} given by
\begin{align*}
q_{t}&=\Big(1+\int_{t-1}^te^{\rho_u}\mu(u)du\Big)^{-1},\quad p_t=e^{\rho_t}q_t,\quad \rho_u:=\rho(t-1,u)=\int_{t-1}^u(\mu(v)-\lambda(v))dv.
\end{align*}
can be embedded in a birth-death process.
Figure \ref{Fig3} illustrates the graphical representation for such an embedding. See also a recent result \cite{FL} presenting a different approach towards dual random forests in a continuous time setting.

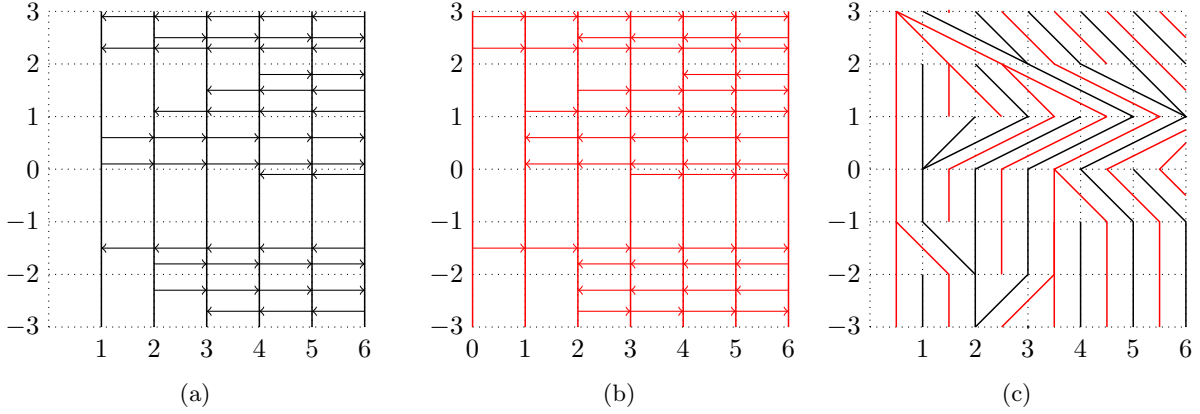
\begin{figure*}[t!]
    \centering
    \begin{subfigure}[t]{0.32\textwidth}
         \begin{tikzpicture}[scale=0.7]
 
\draw[step=1cm,dotted, thin] (0,0) grid (6,6);

\draw [line width=0.2mm,black] (1,0) -- (1,6);
\draw [line width=0.2mm,black] (2,0) -- (2,6);
\draw [line width=0.2mm,black] (3,0) -- (3,6);
\draw [line width=0.2mm,black] (4,0) -- (4,6);
\draw [line width=0.2mm,black] (5,0) -- (5,6);
\draw [line width=0.2mm,black] (6,0) -- (6,6);

\draw [->]  (2,0.7) -- (3,0.7);\draw [->]  (3,0.7) -- (4,0.7);\draw [->]  (4,0.7) -- (5,0.7);\draw [->]  (5,0.7) -- (6,0.7);
\draw [<-]  (3,0.3) -- (4,0.3);\draw [<-]  (4,0.3) -- (5,0.3);\draw [<-]  (5,0.3) -- (6,0.3);

\draw [->]  (2,1.2) -- (3,1.2);\draw [->]  (3,1.2) -- (4,1.2);\draw [->]  (4,1.2) -- (5,1.2);\draw [->]  (5,1.2) -- (6,1.2);
\draw [<-]  (1,1.5) -- (2,1.5);\draw [<-]  (2,1.5) -- (3,1.5);\draw [<-]  (3,1.5) -- (4,1.5);\draw [<-]  (4,1.5) -- (5,1.5);\draw [<-]  (5,1.5) -- (6,1.5);

\draw [<-]  (4,2.9) -- (5,2.9);\draw [<-]  (5,2.9) -- (6,2.9);

\draw [->]  (1,3.1) -- (2,3.1);\draw [->]  (2,3.1) -- (3,3.1);\draw [->]  (3,3.1) -- (4,3.1);\draw [->]  (4,3.1) -- (5,3.1);\draw [->]  (5,3.1) -- (6,3.1);
\draw [->]  (1,3.6) -- (2,3.6);\draw [->]  (2,3.6) -- (3,3.6);\draw [->]  (3,3.6) -- (4,3.6);\draw [->]  (4,3.6) -- (5,3.6);\draw [->]  (5,3.6) -- (6,3.6);

\draw [<-]  (3,4.5) -- (4,4.5);\draw [<-]  (4,4.5) -- (5,4.5);\draw [<-]  (5,4.5) -- (6,4.5);
\draw [<-]  (2,4.1) -- (3,4.1);\draw [<-]  (3,4.1) -- (4,4.1);\draw [<-]  (4,4.1) -- (5,4.1);\draw [<-]  (5,4.1) -- (6,4.1);
\draw [->]  (4,4.8) -- (5,4.8);\draw [->]  (5,4.8)-- (6,4.8);

\draw [<-]  (1,5.9) -- (2,5.9);\draw [<-]  (2,5.9) -- (3,5.9);\draw [<-]  (3,5.9) -- (4,5.9);\draw [<-]  (4,5.9) -- (5,5.9);\draw [<-]  (5,5.9) -- (6,5.9);
\draw [<-]  (1,5.3) -- (2,5.3);\draw [<-]  (2,5.3) -- (3,5.3);\draw [<-]  (3,5.3) -- (4,5.3);\draw [<-]  (4,5.3) -- (5,5.3);\draw [<-]  (5,5.3) -- (6,5.3);
\draw [->]  (2,5.5) -- (3,5.5);\draw [->]  (3,5.5) -- (4,5.5);\draw [->]  (4,5.5) -- (5,5.5);\draw [->]  (5,5.5)-- (6,5.5);

\node[black] () at (-0.5,0){$-3$};
\node[black] () at (-0.5,1){$-2$};
\node[black] () at (-0.5,2){$-1$};
\node[black] () at (-0.3,3){$0$};
\node[black] () at (-0.3,4){$1$};
\node[black] () at (-0.3,5){$2$};
\node[black] () at (-0.3,6){$3$};

\node[black] () at (1,-0.4){$1$};
\node[black] () at (2,-0.4){$2$};
\node[black] () at (3,-0.4){$3$};
\node[black] () at (4,-0.4){$4$};
\node[black] () at (5,-0.4){$5$};
\node[black] () at (6,-0.4){$6$};

\end{tikzpicture}
       \caption{}
    \end{subfigure}
    ~ 
 \begin{subfigure}[t]{0.32\textwidth}
         \begin{tikzpicture}[scale=0.7]

\draw[step=1cm,dotted, thin] (0,0) grid (6,6);

\node[black] () at (-0.5,0){$-3$};
\node[black] () at (-0.5,1){$-2$};
\node[black] () at (-0.5,2){$-1$};
\node[black] () at (-0.3,3){$0$};
\node[black] () at (-0.3,4){$1$};
\node[black] () at (-0.3,5){$2$};
\node[black] () at (-0.3,6){$3$};

\node[black] () at (0,-0.4){$0$};
\node[black] () at (1,-0.4){$1$};
\node[black] () at (2,-0.4){$2$};
\node[black] () at (3,-0.4){$3$};
\node[black] () at (4,-0.4){$4$};
\node[black] () at (5,-0.4){$5$};
\node[black] () at (6,-0.4){$6$};

\draw [line width=0.2mm,red] (0,0) -- (0,6);
\draw [line width=0.2mm,red] (1,0) -- (1,6);
\draw [line width=0.2mm,red] (2,0) -- (2,6);
\draw [line width=0.2mm,red] (3,0) -- (3,6);
\draw [line width=0.2mm,red] (4,0) -- (4,6);
\draw [line width=0.2mm,red] (5,0) -- (5,6);
\draw [line width=0.2mm,red] (6,0) -- (6,6);

\draw [<-,red]  (2,0.7) -- (3,0.7);\draw [<-,red]  (3,0.7) -- (4,0.7);\draw [<-,red]  (4,0.7) -- (5,0.7);\draw [<-,red]  (5,0.7) -- (6,0.7);
\draw [->,red]  (2,0.3) -- (3,0.3);
\draw [->,red]  (3,0.3) -- (4,0.3);\draw [->,red]  (4,0.3) -- (5,0.3);\draw [->,red]  (5,0.3) -- (6,0.3);

\draw [<-,red]  (2,1.2) -- (3,1.2);\draw [<-,red]  (3,1.2) -- (4,1.2);\draw [<-,red]  (4,1.2) -- (5,1.2);\draw [<-,red]  (5,1.2) -- (6,1.2);
\draw [->,red]  (0,1.5) -- (1,1.5);\draw [->,red]  (1,1.5) -- (2,1.5);\draw [->,red]  (2,1.5) -- (3,1.5);\draw [->,red]  (3,1.5) -- (4,1.5);\draw [->,red]  (4,1.5) -- (5,1.5);\draw [->,red]  (5,1.5) -- (6,1.5);

\draw [->,red]  (3,2.9) -- (4,2.9);\draw [->,red]  (4,2.9) -- (5,2.9);\draw [->,red]  (5,2.9) -- (6,2.9);

\draw [<-,red]  (1,3.1) -- (2,3.1);\draw [<-,red]  (2,3.1) -- (3,3.1);\draw [<-,red]  (3,3.1) -- (4,3.1);\draw [<-,red]  (4,3.1) -- (5,3.1);\draw [<-,red]  (5,3.1) -- (6,3.1);
\draw [<-,red]  (1,3.6) -- (2,3.6);\draw [<-,red]  (2,3.6) -- (3,3.6);\draw [<-,red]  (3,3.6) -- (4,3.6);\draw [<-,red]  (4,3.6) -- (5,3.6);\draw [<-,red]  (5,3.6) -- (6,3.6);

\draw [->,red]  (2,4.5) -- (3,4.5);\draw [->,red]  (3,4.5) -- (4,4.5);\draw [->,red]  (4,4.5) -- (5,4.5);\draw [->,red]  (5,4.5) -- (6,4.5);
\draw [->,red]  (1,4.1) -- (2,4.1);\draw [->,red]  (2,4.1) -- (3,4.1);\draw [->,red]  (3,4.1) -- (4,4.1);\draw [->,red]  (4,4.1) -- (5,4.1);\draw [->,red]  (5,4.1) -- (6,4.1);
\draw [<-,red]  (4,4.8) -- (5,4.8);\draw [<-,red]  (5,4.8)-- (6,4.8);

\draw [->,red]  (0,5.9) -- (1,5.9);\draw [->,red]  (1,5.9) -- (2,5.9);\draw [->,red]  (2,5.9) -- (3,5.9);\draw [->,red]  (3,5.9) -- (4,5.9);\draw [->,red]  (4,5.9) -- (5,5.9);\draw [->,red]  (5,5.9) -- (6,5.9);
\draw [->,red]  (0,5.3) -- (1,5.3);\draw [->,red]  (1,5.3) -- (2,5.3);\draw [->,red]  (2,5.3) -- (3,5.3);\draw [->,red]  (3,5.3) -- (4,5.3);\draw [->,red]  (4,5.3) -- (5,5.3);\draw [->,red]  (5,5.3) -- (6,5.3);
\draw [<-,red]  (2,5.5) -- (3,5.5);\draw [<-,red]  (3,5.5) -- (4,5.5);\draw [<-,red]  (4,5.5) -- (5,5.5);\draw [<-,red]  (5,5.5)-- (6,5.5);

\end{tikzpicture}
   
        \caption{}
    \end{subfigure}%
    \begin{subfigure}[t]{0.32\textwidth}
         \begin{tikzpicture}[scale=0.7]

\draw[step=1cm,dotted, thin] (0,0) grid (6,6);

\draw [line width=0.2mm,black] (3,-0.03) -- (3,0.03);

\draw [line width=0.2mm,black] (1,0) -- (1,1);
\draw [line width=0.2mm,black] (2,0) -- (2,1) -- (1,2) -- (1,3) -- (1,4) -- (1,5);
\draw [line width=0.2mm,black] (1,3) -- (2,4);
\draw [line width=0.2mm,black] (1,3) -- (3,4)-- (2,5);
\draw [line width=0.2mm,black] (2,1) -- (2,2) -- (2,3)-- (4,4);
\draw [line width=0.2mm,black] (2,0) -- (3,1) -- (3,2) -- (3,3) -- (5,4)-- (3,5)-- (1,6);
\draw [line width=0.2mm,black] (3,5)-- (2,6);
\draw [line width=0.2mm,black] (4,0) -- (4,1) -- (4,2);
\draw [line width=0.2mm,black] (5,0) -- (5,1) -- (5,2) -- (4,3) -- (6,4)-- (4,5)-- (3,6);
\draw [line width=0.2mm,black] (6,4)-- (5,5)-- (4,6);
\draw [line width=0.2mm,black] (6,0) -- (6,1) -- (6,2) -- (5,3);
\draw [line width=0.2mm,black] (6,5)-- (5,6);

\node[black] () at (-0.5,0){$-3$};
\node[black] () at (-0.5,1){$-2$};
\node[black] () at (-0.5,2){$-1$};
\node[black] () at (-0.3,3){$0$};
\node[black] () at (-0.3,4){$1$};
\node[black] () at (-0.3,5){$2$};
\node[black] () at (-0.3,6){$3$};

\node[black] () at (1,-0.4){$1$};
\node[black] () at (2,-0.4){$2$};
\node[black] () at (3,-0.4){$3$};
\node[black] () at (4,-0.4){$4$};
\node[black] () at (5,-0.4){$5$};
\node[black] () at (6,-0.4){$6$};


%
\draw [line width=0.2mm,red] (1.5,6.03) -- (1.5,5.97);
\draw [line width=0.2mm,red]  (0.5,0)-- (0.5,6);
\draw [line width=0.2mm,red]  (1.5,0)-- (1.5,1)-- (0.5,2);
\draw [line width=0.2mm,red]  (1.5,2)-- (1.5,3)-- (3.5,4)-- (2.5,5)-- (0.5,6);
\draw [line width=0.2mm,red]  (2.5,4)-- (1.5,5)-- (0.5,6);
\draw [line width=0.2mm,red]  (1.5,4)-- (1.5,5);
\draw [line width=0.2mm,red]  (2.5,1)-- (2.5,3)-- (4.5,4);
\draw [line width=0.2mm,red]  (2.5,0)-- (3.5,1)-- (3.5,3);
\draw [line width=0.2mm,red]  (3.5,0)-- (3.5,2);
\draw [line width=0.2mm,red]  (4.5,4)-- (2.5,5);
\draw [line width=0.2mm,red]  (4.5,5)-- (3.5,6);
\draw [line width=0.2mm,red]  (4.5,0)-- (4.5,2)-- (3.5,3)-- (5.5,4)-- (3.5,5)-- (2.5,6);
\draw [line width=0.2mm,red]  (5.5,0)-- (5.5,2)-- (4.5,3)-- (6,3.75);
\draw [line width=0.2mm,red]  (6,2.5)-- (5.5,3)-- (6,3.5);
\draw [line width=0.2mm,red]  (6,4.5)-- (4.5,6);
\draw [line width=0.2mm,red]  (6,5.5)-- (5.5,6);
\end{tikzpicture}
       \caption{}
    \end{subfigure}
    \caption{Graphical representation of the birth-death processes. Arrows to the left mean the death at the end position of the line of arrows. Arrows to the right mean splitting at the origin of the line of arrows.}
\label{Fig3}\end{figure*}

\vspace{0.4cm}\noindent
{\bf Defective rank-dependent GW.}  For any $V\in\Phi_0$, there is either finite or infinite limit $\bar V=\lim_{x\to\infty} V(x)$. We will call {\it defective} a random reproduction mapping $U\in\Phi_0$ such that  its dual $V$ satisfies $\rP(\bar V<\infty)>0$. In the defective case, a particle is able to produce infinitely many offspring. GW processes with a defective reproduction law
were studied in a recent paper \cite{SC}.

Turning to the non-linear birth-death processes, see for example \cite{SS}, we observe that in general, the embedding, discussed in the previous example, does not yield a rank-dependent GW process, as the number of offspring may depend on each other.  
An interesting exception is the pure death processes producing embedded pure death rank-dependent GW processes mentioned in the first example of this section. Given the time-homogeneous death rate $\mu_x$ for an individual of rank $x\ge1$, such that
\[\sum_{x=1}^\infty {1\over \mu_1+\ldots+\mu_x}<\infty,\]
we get a pure death process coming down from infinity, see for example \cite{SF}. Observe that in this case, the dual Markov chain gives a defective reproduction model which is not a rank-dependent GW process. 

\section{Proofs}\label{Spro}
Let reproduction mappings $U,V,\tilde U\in\Phi_0$, 
be connected as  $V=U^-$,  $\tilde U=V^-$. The corresponding offspring numbers are denoted $u(x), v(x)$ and $\tilde u(x)$.

  \begin{lemma}\label{L}
If $(\xi_i,\eta_i)$ are defined by
 \begin{align}
 (u(1),u(2),\ldots)&=(\underbrace{0,\ldots,0}_{\xi_1},\eta_{1}+1,\underbrace{0,\ldots,0}_{\xi_2},\eta_{2}+1,\ldots),\quad \xi_i,\eta_i\ge0,\quad i\in\mathbb N, \label{hata1}
\end{align}
then we have \eqref{hata} and 
 \begin{align}
 (\tilde u(1),\tilde u(2),\ldots)&=(1,\underbrace{0,\ldots,0}_{\xi_1},\eta_{1}+1,\underbrace{0,\ldots,0}_{\xi_2},\eta_{2}+1,\ldots) .\label{hatb}
\end{align}
\end{lemma}

\begin{proof}
From  $V(x)=\min\{y:U(y)\ge x\}$, we get $V(0)=0$ and
\begin{align}
\{x: V(x)=y\}= \{x: U(y-1)<x\le U(y)\},\quad y\ge1,\label{vx}
\end{align}
which implies  \eqref{hata}. In a similar way, relation \eqref{hatb} follows from \eqref{hata}.
Observe also that  \eqref{vx} entails \eqref{VxU}.
\end{proof}

\vspace{0.4cm}\noindent
{\bf Proof of Theorem \ref{proM}.}  
Using \eqref{VxU}, we obtain consecutively
\begin{align*}
\{\hat U_{-b,-a}(x)\le y\}&=\{\hat U_{-a-1}\circ\cdots\circ \hat U_{-b}(x)\le y\}=\{V_{a}\circ\cdots\circ V_{b-1}(x)\le y\}\\
&=\{V_{a+1}\circ\cdots\circ V_{b-1}(x)\le U_a(y)\}=\ldots=\{x\le U_{b-1}\circ\cdots\circ U_a(y)\}=\{x\le U_{a,b}(y)\}.
\end{align*}
Observe that $\tilde U_t=V_t^-$, and by  Lemma \ref{L}, we have $\tilde U_t(x)=U_t(x-1)+1$, which yields
\begin{align*}
\tilde U_{a,b}(x)&=\tilde U_{b-1}\circ\cdots\circ \tilde U_{a}(x)=\tilde U_{b-1}\circ\cdots\circ \tilde U_{a+1}(U_a(x-1)+1)\\
&= \tilde U_{b-1}\circ\cdots\circ \tilde U_{a+2}( U_{a+1}\circ U_a(x-1)+1)= U_{a,b}(x-1)+1.
\end{align*}

\vspace{0.4cm}\noindent
{\bf Proof of  Proposition \ref{JJ}.}  
The random variables $u(1),u(2),\ldots$ are independent with a common  distribution $\{P_k\}_{k=0}^\infty$ if and only if relation   \eqref{hata1} holds with mutually independent $\xi_1,\eta_1,\xi_2,\eta_2,\ldots$, such that
\begin{align*}
 \rP( \xi_i=k)=P_{0}^k(1-P_{0}),\quad
 \rP(\eta_i=k)=\rP(u(1)=k+1|u(1)\ge1),\quad k\ge0,\quad i\ge1.
\end{align*}
By Lemma \ref{L}, this proves the first statement of the proposition. 

Turning to the second statement concerning the distribution of $v(x)$, denote $\hat q(x)=\rP(v(x)>0)$. The first statement implies that \eqref{dlf} holds with 
\begin{align*}
\hat q(x)&= \rP\Big(\bigcup_{n=0}^\infty \{1+n+\eta_{1}+\ldots+\eta_{n}=x\}\Big)=\sum_{n=0}^{x-1} \rP(\eta_{1}+\ldots+\eta_{n}=x-n-1),\quad x\ge1.
\end{align*}
This entails that $\hat q(1)=1$, and  recursion \eqref{p0x} for $x\ge2$, is obtained via conditioning on $\eta_1$:
\begin{align*}
\hat q(x)&=\sum_{n=1}^{x-1} \rP(\eta_{1}+\ldots+\eta_{n}=x-n-1)=\sum_{n=1}^{x-1}\sum_{k=1}^{x-n} \rP(\eta_{1}=k-1)\rP(\eta_{2}+\ldots+\eta_{n}=x-n-k)\\
&=(1-P_0)^{-1}\sum_{k=1}^{x-1} P_k\sum_{n=1}^{x-k}\rP(\eta_{2}+\ldots+\eta_{n}=x-k-n)=(1-P_0)^{-1}\sum_{k=1}^{x-1} P_k\hat q(x-k).
\end{align*}

\vspace{0.4cm}\noindent
{\bf Proof of  Theorem \ref{th}.}  
Suppose that the conditions of Proposition \ref{JJ} are valid. 

If, as stated by \eqref{lft}, $\hat q(x)=q$ for all  $x\ge2$, then Proposition \ref{JJ} implies  \eqref{lfg} it is easy to verify independence of 
$v(1),v(2),\ldots$. Thus we find that the dual reproduction is that of a GW process with an eternal particle. 

To prove the reverse statement assume that $v(1),v(2),\ldots$ are independent and $v(2),v(3),\ldots$ have a common distribution. Using \eqref{p0x}, we find that for some $q\in(0,1]$,
\[ (1-P_{0})q=p(x-1)+q \sum _{k=1}^{x-2}P_{k},\quad x\ge2.\]
Therefore, for all $n\ge1$, we obtain 
\[ p(n)=q \sum _{k=n}^\infty P_{k},\]
which leads to \eqref{lfg}, which in turn yields \eqref{lft}.

\vspace{0.4cm}\noindent{\bf Acknowledgements.} The authors thank Uwe R\"{o}sler for a discussion of an issue concerning duality. The research by Jonas Jagers was supported by the Royal Swedish Academy of Sciences through the Elis Sidenbladh foundation grant.

\end{document}